\documentclass[11pt,a4paper, envcountsame,mathserif]{amsart}
\usepackage[usenames,dvipsnames]{color}

\usepackage[utf8]{inputenc}	
\usepackage{pdfpages}

 \usepackage[colorlinks,citecolor=blue,urlcolor=blue, linkcolor=blue, backref]{hyperref}

\usepackage[capitalise]{cleveref}		

 \usepackage{amsmath, amssymb, xspace}
 \usepackage{graphicx}

 \usepackage[all]{xy}

\usepackage{tikz}
\usetikzlibrary{arrows,snakes,positioning,backgrounds,shadows}
\usepackage{stmaryrd}

\newtheorem{theorem}{Theorem}[section]

\newtheorem{thm}[theorem]{Theorem}

\newtheorem{proposition}[theorem]{Proposition}

\newtheorem{claim}[theorem]{Claim}

\newtheorem{lemma}[theorem]{Lemma}

\newtheorem{question}[theorem]{Question}

\theoremstyle{definition}
\newtheorem{definition}[theorem]{Definition}
\newtheorem{example}[theorem]{Example}
\newtheorem{remark}[theorem]{Remark}

\newcommand{\NN}{{\mathbb{N}}}

\newcommand{\QQ}{{\mathbb{Q}}}
\newcommand{\ZZ}{{\mathbb{Z}}}
\newcommand{\sub}{\subseteq}
\newcommand{\sN}[1]{_{#1\in \NN}}
\newcommand{\uhr}[1]{\! \upharpoonright_{#1}}

\newcommand{\bi}{\begin{itemize}}
\newcommand{\ei}{\end{itemize}}
\newcommand{\bc}{\begin{center}}
\newcommand{\ec}{\end{center}}

\newcommand{\ex}{\exists}

\newcommand{\la}{\langle}
\newcommand{\ra}{\rangle}

\newcommand{\n}{\noindent}

\newcommand{\vsp}{\vspace{6pt}}

\newcommand{\sss}{\sigma}
\newcommand{\aaa}{\alpha}

\newcommand{\lland}{\, \land \, }

\newcommand \seq[1]{{\left\langle{#1}\right\rangle}}

\newcommand{\ol}{\overline}

\newcommand{\RA}{\Rightarrow}
\newcommand{\LA}{\Leftarrow}

\newcommand{\rapf}{\n $\RA:$\ }
\newcommand{\lapf}{\n $\LA:$\ }

\DeclareMathOperator{\Ext}{\mathrm{Ext}}
\DeclareMathOperator{\Hom}{\mathrm{Hom}}

\numberwithin{equation}{section}
\renewcommand{\hat}{\widehat}
\begin{document}

\title{Logic Blog 2021}

 \author{Editor: Andr\'e Nies}

\maketitle


\setcounter{tocdepth}{1}
\tableofcontents

 {
The Logic   Blog is a shared platform for
\bi \item rapidly announcing  results and questions related to logic
\item putting up results and their proofs for further research
\item parking results for later use
\item getting feedback before submission to  a journal
\item fostering collaboration.   \ei

Each year's   blog is    posted on arXiv  2-3 months after the year has ended.
\vsp
\begin{tabbing} 

 \href{https://arxiv.org/abs/2101.09508}{Logic Blog 2020} \ \ \ \   \= (Link: \texttt{http://arxiv.org/abs/2101.09508})  \\
 
 \href{http://arxiv.org/abs/2003.03361}{Logic Blog 2019} \ \ \ \   \= (Link: \texttt{http://arxiv.org/abs/2003.03361})  \\

 \href{http://arxiv.org/abs/1902.08725}{Logic Blog 2018} \ \ \ \   \= (Link: \texttt{http://arxiv.org/abs/1902.08725})  \\
 
 \href{http://arxiv.org/abs/1804.05331}{Logic Blog 2017} \ \ \ \   \= (Link: \texttt{http://arxiv.org/abs/1804.05331})  \\
 
 \href{http://arxiv.org/abs/1703.01573}{Logic Blog 2016} \ \ \ \   \= (Link: \texttt{http://arxiv.org/abs/1703.01573})  \\
 
  \href{http://arxiv.org/abs/1602.04432}{Logic Blog 2015} \ \ \ \   \= (Link: \texttt{http://arxiv.org/abs/1602.04432})  \\
  
  \href{http://arxiv.org/abs/1504.08163}{Logic Blog 2014} \ \ \ \   \= (Link: \texttt{http://arxiv.org/abs/1504.08163})  \\

   \href{http://arxiv.org/abs/1403.5719}{Logic Blog 2013} \ \ \ \   \= (Link: \texttt{http://arxiv.org/abs/1403.5719})  \\

    \href{http://arxiv.org/abs/1302.3686}{Logic Blog 2012}  \> (Link: \texttt{http://arxiv.org/abs/1302.3686})   \\

 \href{http://arxiv.org/abs/1403.5721}{Logic Blog 2011}   \> (Link: \texttt{http://arxiv.org/abs/1403.5721})   \\

 \href{http://dx.doi.org/2292/9821}{Logic Blog 2010}   \> (Link: \texttt{http://dx.doi.org/2292/9821})  
     \end{tabbing}

\vsp

\n {\bf How does the Logic Blog work?}

\vsp

\n {\bf Writing and editing.}  The source files are in a shared dropbox.
 Ask Andr\'e (\email{andre@cs.auckland.ac.nz})  in order    to gain access.

\vsp

\n {\bf Citing.}  Postings can be cited.  An example of a citation is:

\vsp

\n  H.\ Towsner, \emph{Computability of Ergodic Convergence}. In  Andr\'e Nies (editor),  Logic Blog, 2012, Part 1, Section 1, available at
\url{http://arxiv.org/abs/1302.3686}.}

%

%

 
The logic blog,  once it is on  arXiv,  produces citations e.g.\ on Google Scholar.
%

%

  \part{Group  theory and its connections to  logic}
   
   \section{Berdinsky and Nies: Abelian groups of rank 2}

     A set $\{G_i\}_{i \in I}$ of torsion--free    
     groups $\neq 0$ is said to be a rigid system 
     if $\mathrm{Hom}(G_i, G_j)$ is a subgroup
     of $\mathbb{Q}$ if $i = j$ or $0$ if 
     $i \neq j$. The groups in a rigid system 
     are necessarily indecomposable. 
     A group  $G$ is called rigid if 
     the singleton $\{G\}$ is rigid.  
 \subsection{Background}
  \begin{lemma}[Lemma 4.6 in \cite{Fuchs:15}]  
  \label{Lemma_4.6_Fuchs}	  
  	  For every $r \geqslant 2$, there 
  	  exists a rigid system of 
  	  $2^{\aleph_0}$ torsion--free groups 
  	  $\{A_i \}_{i \in I}$ of rank $r$ such that
  	  $\mathrm{End}\, A_i \cong (\mathbb{Z}, +)$ 
  	  for each $i \in I$. The groups 
  	  are homogeneous of type  
  	  ${\bf t}_0 = (0,\dots,0,\dots)$,  (i.e., every rank 1 subgroup is free). 
  \end{lemma}	
  \begin{proof} 
  	 Each of these groups  is  
  	 defined as a subgroup of 
  	 $$\mathbb{Z}[1/p] a_1 \oplus \dots 
  	  \oplus \mathbb{Z}[1/p] a_r.$$ 
    Here we consider only the case $r=2$. 
   Following Fuchs,  choose a $p$--adic unit 
    $$ 
        \pi = s_0 + s_1 p + \dots + s_n p^n + \dots,   
    $$
    where $0 < s_0 <p$ and $0 \leqslant s_i < p$ for 
    all $i>0$, which is transcendental over 
    $\mathbb{Z}$ (or, at least, satisfies no quadratic equation).      Let   for $n \geqslant 1$
    \begin{equation}
    \label{formula_for_x_n}   
       x_n = p^{-n} (a_1 + \pi_n a_2) \in G, 
    \end{equation}	
    where $\pi_n = s_0 + s_1 p + \dots + s_{n-1} p^{n-1}$.  
    The subgroup $A = A_\pi \leqslant G$ is defined by 
    $$
       A_\pi = \langle a_1, a_2; x_1, x_2, \dots x_n, \dots \rangle.
    $$
    The group $A$ is of rank $2$. It is then showed 
    that $G$ is rigid and 
    $\mathrm{End}\,A \cong \mathbb{Z}$. 
    We skip the proof of this fact. 
    Now we prove that $A$ is of type ${\bf t}_0$. 
    Since $p x_{i+1} = x_i + s_i a_2$ for every 
    $i\ge 1$, an element of $A$ not in 
    $B = \langle a_1, a_2 \rangle$ has the form 
    $k x_n + k_2 a_2$ for some integers $k, k_2$ 
    with $\gcd (p,k)=1$, and for 
    some $n \geqslant 1$.  
    
        
    Furthermore, if 
    $p^m (k x_n + k_2 a_2) = \ell a_2$ for 
    $\ell \in \mathbb{Z}$, then it follows from 
    \eqref{formula_for_x_n} that the coefficient  
    of $a_1$ must be zero. Therefore, $\langle a_2 \rangle$ 
    is a $p$--pure subgroup of $A$. 
    Assume that $p^{-n} (k_1 a_1 + k_2 a_2) \in A$ for
    all $n > 0$; here $k_1 \neq 0$, since 
    $p^n | a_2$ is impossible; the latter is 
    because $\langle a_2 \rangle$ 
    is $p$--pure. Then also $p^{-n} (k_1 a_1 + k_2 a_2 - 
    k_1 p^n x_n) \in A$, so by \eqref{formula_for_x_n} 
    and the purity of $\langle a_2 \rangle$ we get 
    that $p^n | k_1 - k_2 \pi_n$ for all $n$. This 
    implies that $\pi$ is a rational number, a 
    contradiction. So $A$ is homogeneous of 
    type ${\bf t}_0$. 
  \end{proof}	
 \begin{question} 
 	Is there some $\pi$ not satisfying any quadratic equation with integer coefficients such that $A$ is FA--presentable?  
 \end{question}	
 \subsection{Group extensions and   cocycles} 
 
 For more detail on this see the next section, with Lupini.

  Let $0\to A\xrightarrow{\mu} E \xrightarrow{\nu}C \to 0$
  be a short exact sequence of abelian groups. 
  The group $E$ is called an extension of $C$ 
  by $A$.  
  Let 
  	      $f : C \times C \to A$ be a cocycle 
          on $C$ to $A$, that is a  map satisfying 
          the properties: 
          $$f(u,0) = 0, f(u,v)=f(v,u), f(u,v)+f(u+v,w)=f(u,v+w)+f(v,w)$$
          for all $u,v,w \in C$.

          The 
          extension of $C$ by $A$ corresponding to 
          $f$ is constructed as the set of all pairs 
          $(u, a) \in C \times A$ with the operation:
          $$(u,a) + (v,b) : = (u+ v, a+ b + f(u,v)).$$
          Clearly, if $A$ and $C$ are FA--presentable and the graph of 
          $f$ is FA--recognizable for appropriate such presentations, then  the corresponding extension           is FA--presentable.

  To obtain a cocycle for a particular extension $E$ of $A$ by $C$, one chooses a canonical transversal set $T\subseteq E$, i.e., $T$ meets each coset of $\mu A$ on a single element, and $0 \in T$. Then the surjection $\nu: E \rightarrow C$ restricted 
  on $T$ is a bijection.  
  Let $\phi \colon C \to T$ be a 
  bijection for which $\nu \circ \phi = id$.    
  We call $\phi$ a transversal map. 
  The corresponding cocycle $f\colon C \times C \to A$ is given by \[f(u,v) = \phi(u) + \phi(v) - \phi(u+v).\] This intuitively measures the deviation of $\phi$ from being a homomorphism, or in other words, the deviation of $T$ from being a subgroup of $E$. 
  
  \subsection{Groups $A_\pi$ as Extensions of $\mathbb Z$ by $\mathbb Z[1/p]$}
  
     Let us define a short exact sequence: $$0\to\mathbb{Z}\xrightarrow{\mu}A_\pi \xrightarrow{\nu}
        \mathbb{Z}[1/p]\to 0$$
     as follows.
     An embedding $\mu:\mathbb{Z}\rightarrow A_\pi$ is defined by the equation: $\mu(m)=m a_{2}$, $m\in\mathbb{Z}$.
     Recall that each $y\in A_\pi$, which is not in $\langle a_{1},a_{2}\rangle$, admits a unique form $kx_{n}+ma_{2}$ for 
     $n \geqslant 1$ and some integers
     $m$ and $k$ for which $\gcd(p,k)=1$, see Lemma 
     \ref{Lemma_4.6_Fuchs}.
     A surjection $\nu:A_\pi \rightarrow\mathbb{Z}[1/p]$ 
     is defined as:
     \begin{itemize}
     	\item{if $y=ka_{1}+ma_{2}$, then $\nu(y)=\nu(ka_{1}+ma_{2})=k$;}
     	\item{if $y\notin\langle a_{1},a_{2}\rangle$, then $\nu(y)=\nu(kx_{n}+ma_{2})=\frac{k}{p^{n}}$.}
     \end{itemize}
     The map $\nu$ is a homomorphism (it can be verified directly). The same surjective homomorphism $\nu$ 
     can be defined as follows. 
     Each element $y\in A_\pi$ admits the following 
     unique normal form: $y=ma_{1}+ka_{2}+\sum_{i=1}^{n}r_{i}x_{i}$, where $0\leqslant r_{i}<p$ for $i=1,\dots,n$.
     This follows from the identity $px_{i+1}=x_{i}+s_{i}a_{2}$.
     The map $\nu : A_\pi \rightarrow \mathbb{Z}[1/p]$ can be defined as: $\nu(y)=m+\frac{r_{1}}{p}+\frac{r_{2}}{p^{2}}+\dots+\frac{r_{n}}{p^{n}}$. Clearly,  it gives exactly the same surjective homomorphism from $A_\pi$ to $\mathbb{Z}[1/p]$ as above. 
  
     Note that a cocycle depends on the choice of a transversal map $\phi$. Below we consider different
     transversal maps and the corresponding cocycles.   

  {\bf First version:}    

 Below we will describe an appropriate cocycle for $A_\pi$.  Recall that $\pi= \sum_n s_n p^n$ is a fixed $p$-adic unit.
 As  a  canonical transversal set $T$ choose the set of elements of the form $m a_1 + \sum_{i = 1}^n r_i x_i$,  where $ m \in \mathbb Z$, $n \in \mathbb N$, and $0 \le r_i < p$ for each 
 $i=1,\dots,n$. 
 Let us define a transversal map 
 $\phi: \mathbb{Z}[1/p] \rightarrow A_\pi$.  
 For $u= m.r_1 r_2 \ldots r_n \in \mathbb Z[1/p]$,   
 let $\phi(u)= m a_1 + \sum_{i = 1}^n r_i x_i$.  Given $u,v \in \mathbb Z[1/p]$ one can write   $u= m.r_1 r_2 \ldots r_n \in \mathbb Z[1/p]$ and $v= k. t_1 \ldots t_\ell$. One
 can calculate  the corresponding cocycle
 $f(u,v)= \phi(u)+\phi(v)-\phi(u+v)$ as follows.   
 Assuming that $n = \ell$, by adding a tail with zeros  when necessary, we obtain: 
    
    \[ f(u,v) = \sum_{0 \le  i \le n-1 \, \land \,  i+1\in F} s_i,  \]
  where $F$  is the set of positions where one has  is carry when performing  the usual addition algorithm,   using that 
  $p x_{i+1} = x_{i} + s_{i} a_2$ for every 
    $i\ge 2$ and $px_1 = a_1 + s_{0}a_2$.  For example,  if $p=3$, $u= 0.112$ and $v= 0.211$, then $F= \{1,2,3\}$. 
    
    
    Even in the simple case that all $r_i= 1$ (and hence the $p$-adic number $\pi$ is a rational), this doesn't lead naturally to an FA presentation. We assume that $\mathbb{Z}[1/p]$ is represented similar to \cite{Nies.Semukhin:09}. Then, to compute $f$, an FA  would have  to count the number of $i$ satisfying the given condition $ r_i + t_i \ge p$.

   {\bf Second version:} 
   
   Alternatively, a transversal map $\phi:\mathbb{Z}[1/p]\rightarrow A_\pi$ can be defined as:
   \begin{itemize}
   \item{if $z=k$ for some integer $k$, we define $\phi(z)=ka_{1}$;}
   \item{if $z=\frac{k}{p^{n}}$ for $n \geqslant 1$ and $\gcd(k,p)=1$, we define $\phi(z)=kx_{n}$.}
   \end{itemize}
   The corresponding cocycle 
   $f: \mathbb{Z}[1/p] \times \mathbb{Z}[1/p] \rightarrow 
    \mathbb{Z}$  
   defined by the equation 
   $\mu f (u,v) = \phi(u) + \phi(v) - \phi(u+v)$
   is then as follows: 
   \begin{itemize}
   \item{if $u,v \in \mathbb{Z}$, then 
         $\mu f(u,v)= ua_1 + v a_1 - (u+v)a_1 =0$. 
         Therefore, $f(u,v)=0$.}
   \item{if $u \in \mathbb{Z}$ and 
   	        $v = \frac{k}{p^n}$ for $\gcd(k,p)=1$ and 
   	         $n \geqslant 1$, then $u+v = \frac{up^n+ k}{p^n}$. 
         Therefore, $\mu f(u,v)=u a_1 + k x_n - 
         (up^n+ k) x_n = u a_1 - u (a_1 + \pi_n a_2)  =
         -u\pi_n a_2$. Therefore, $f(u,v)=-u \pi_n$.}
    \item{if $u= \frac{k_1}{p^m}$ and 
             $v= \frac{k_2}{p^n}$ for $1 \leqslant m<n$, 
          then $\frac{k_1}{p^m} + \frac{k_2}{p^n} = 
           \frac{k_1 p^{n-m} + k_2}{p^n}$. 
       Therefore, $\mu f(u,v) = k_1 x_m + k_2 x_n - 
        (k_1 p^{n-m} + k_2)x_n = k_1 x_m - k_1 p^{n-m} x_n
        = k_1 p^{-m} (a_1 + \pi_m a_2) - k_1 p^{-m}
         (a_1 + \pi_n a_2) = \frac{k_1}{p^m} 
         (\pi_m - \pi_n) a_2= u (\pi_m - \pi_n) a_2$.
          Therefore, $f(u,v) = u (\pi_m - \pi_n)$.}
     \item{if $u=\frac{k_1}{p^n}$ and $v = \frac{k_2}{p^n}$ for
     	$n \geqslant 1$, then 
     	$\frac{k_1}{p^n} + \frac{k_2}{p^n} = 
     	 \frac{p^m k}{p^n}= \frac{k}{p^{n-m}}$ 
     	 either for some $0 \leqslant m < n$ and 
     	 $k$ for which $\gcd(k,p)=1$ or $m=n$, that is, $\frac{k_1}{p^n} + \frac{k_2}{p^n}  = k$.
     	
     	 In the first case we have: 
     	 $\mu f(u,v) = k_1 x_n + k_2 x_n - k x_{n-m}=
     	 (k_1 +k_2)p^{-n}(a_1 + \pi_n a_2) - 
     	 k p^{m-n} (a_1 + \pi_{n-m} a_2) =
     	 (u+v)(\pi_n -  \pi_{n-m})a_2$, 
     	 so $f(u,v) =(u+v)(\pi_n -  \pi_{n-m})$. 
         
         In the second case we have: 
         $\mu f(u,v) = k_1 x_n + k_2 x_n - k a_1 =
         (k_1 + k_2) p^{-n}(a_1 + \pi_n a_2) - k a_1 =
         (u+v)\pi_n a_2$, so $f(u,v) = (u+v) \pi_n$.}   	 
   \end{itemize} 

   \section{Lupini and Nies: Extensions of abelian groups}

Eilenberg and MacLane~\cite{Eilenberg.MacLane:42} were   among the first to study extensions of abelian groups via the cohomological notions of cocycles and coboundaries. Here  we will expose the  algorithmic content of some of their results. Thereafter,  we use this to show that \bc $\ZZ_p \cong \Ext(\ZZ(p^\infty), \ZZ)$ and $\ZZ_p/\ZZ \cong \Ext(\ZZ[1/p], \ZZ)$ \ec via isomorphisms that are computable in the sense of the Weihrauch setting of computable analysis.  (As usual, $\ZZ_p$   denotes the group of $p$-adic integers, and    $\ZZ(p^\infty)= \ZZ[1/p]/ \ZZ$ the Pr\"ufer group for $p$.)
\subsection{Some background and preliminaries} All groups will be  abelian. Given groups $A,C$, an \emph{extension} of $C$ by $A$ is an exact sequence 
$0 \to A \to E \to C \to 0$. Extensions can be described   by \emph{cocycles}, that is,  functions $f \colon C \times C \to A$ that are symmetric and satisfy the  condition that \begin{equation} \label{eqn:coc} f(u,v)+ f(u+v, w)=f(v,w) + f(u,v+w).\end{equation} By these conditions,  on $C \times A$, the operation 
\[ (u,a) + (v,b) : = (u+ v, a+ b + f(u,v))\]
defines an abelian group $E= E_f$. 
This operation is commutative because $f$ is symmetric.  The inverse of $(u,a)$ is $(-u.-a-f(u,-u))$. For associativity,   if we calculate $[(u,a) + (v.b)]+ (w,c)$, the ``correcting term" in the second component on the right side is $f(u,v) + f(u+v,w)$. If we calculate $(u,a) + [(v.b)+ (w,c)]$, the correcting term is $f(v,w)+ f(u,v+w)$.

  The  cocycles from $C$ to $A$  form a group under addition, which  is denoted $Z(C,A)$. A \emph{coboundary} is a cocycle $g$ of the form $g(u,v)= \varphi(u)+ \varphi(v)- \varphi(u+v)$ where $\varphi \colon C \to A$ is  any   function with $\varphi(0)=0$. The  group of coboundaries is denoted $B(C,A)$.  These correspond to splitting extensions (i.e., direct products in the abelian setting). Two cocycles are called equivalent if their difference is a coboundary.  The group of extensions of $C$ by $A$ is \bc $\Ext(C,A)= Z(C,A)/ B(C,A)$. \ec   
  Note that one can add to $f$   a constant in $A$ without invalidating the condition~(\ref{eqn:coc}). Also,  every constant  function   $C\times C \to A$ is  a co-boundary. So we may assume w.l.o.g.\  that $f$ is normalized in the sense that $f(0,0)= 0$. By~(\ref{eqn:coc}) setting $u=v=0$ it follows that  $f(0,w)=f(w,0)= 0$ for each $v$. If $f$ is normalized,   one obtains an exact sequence 
$0 \to A \to E \to C \to 0$, $E= E_f$, where  the embedding $A\to E$ is given by $a \to(0,a)$, and the     projection map $E \to C$ is given by  $(c,a) \to c$. 

For detail see e.g.\  Fuchs~\cite{Fuchs:15} (who calls these objects   ``extensions of $A$ by $C$",  but uses the same order, first  $C$ then $A$,  in the  notation).

 
 \subsection{Computable analysis}
 If $A,C$  are computable groups (thus countable, in particular),  the  computable cocycles form a subgroup of $Z(C,A)$. They yield exact sequences as above where $E$ and also the connecting maps  are computable.  
 
 One of the questions we     ask is  whether the groups $Z(C,A)$ and  $\Ext(C,A)$, which may well be uncountable, are  computable in the sense of  computable analysis. We review the   setting going back to Weihrauch, based on the tutorial \cite{Brattka.Hertling.ea:08}, i.p.\  their Def.\ 4.2.

 Let $\Sigma$ be an alphabet (which may be infinite). A representation  of a set $X $ is an onto map $\delta \colon \sub \Sigma^\omega \to X$. One calls $(X, \delta,\Sigma)$ a represented set. If $\delta(\aaa)=x$ one calls $\aaa$ a $\delta$-name for $x$.  A function $F $ from $X$ to $X'$ is called computable if from any name for $x$  one can compute a name for $F(x)$. The formal definition is as follows.
 \begin{definition} \label{df:names} Let $(X, \delta, \Sigma)$ and $X', \delta', \Sigma')$ be represented sets. A function $F \colon  X \to X'$ is called \emph{computable} if there is an oracle Turing machine $M$ using $\delta$-names as oracles such that for each $\delta$-name $\aaa$, one has that $M^\aaa$ is a $\delta'$-name,  and 
 \bc $F(\delta(\aaa))  = \delta'(M^\aaa))$. \ec\end{definition}
 
 \begin{example} \label{ex:Zp} Let $X= \ZZ_p$ for some prime $p$. Let $\Sigma= \ZZ$. For $\aaa \in \Sigma^\omega$, let $\delta(\aaa)= \sum_n \aaa(n)p^n$.
 \end{example}
  \subsection{A description of $\Ext(H,A)$ by homomorphism groups}
The following important lemma of Eilenberg and MacLane elaborates on the  easy fact that each extension of a free group  $F$ splits, or in other words, every cocycle on $F$ is a coboundary.

\begin{lemma}[Lemma 7.3 in~\cite{Eilenberg.MacLane:42}] {\rm  Let $F$ be a free group, with generators $z_\aaa$. Let the cocycle  $h \in Z(F,A)$ be normalized. Then there is a function $\varphi \colon F \to A$ such that \begin{equation} \label{hhh} h(x,y)=  \varphi(x+y)  -\varphi(x) - \varphi(y)  , \end{equation} 
that is,    $h$ is a coboundary via  $-\varphi$.
Moreover, one can ensure  that $\varphi(0)= 0$ and $\varphi(z_\aaa)=0$ for each $\aaa$. 

 The \emph{computable setting} is that  $F$ is computable, the map $\aaa \to z_\aaa $ is computable and has computable range, and $h$ is computable. Then $\varphi$ is computable as well. }
 \end{lemma} 
 
 \begin{proof} Let $E$ be the extension given by a normalized cocycle $h$. Let $\beta \colon E \to F$ be the projection. Corresponding to $h$ there is a set of coset representatives. For $x \in F$ let $u'(x)$ be the representative  in $E$, so that $\beta(u'(x)) = x$.  Since  $h$ is normalized, we have  $u'(0)= 0$.  We have   
 \begin{equation} \label{eqn: hphi} h(x,y) = u'(x)+ u'(y) -u'(x+ y)    \end{equation}  (such a relation is called an addition table in~\cite{Eilenberg.MacLane:42}).  Note that  with the concrete construction of $E$ from $h$  and   the embedding $A \to E$ by $a \mapsto (0,a)$ as   above, we have $u'(x) = (x,0)$.

 Now, each  element $x$ of $F$ is given uniquely in the form $x=\sum e_\aaa z_\aaa$, with all but finitely many $e_\aaa \in \ZZ$ being $0$.  Define a new system of representatives by \bc $u(x)= \sum e_\aaa u'(z_\aaa)$.  \ec Since $F$ is freely generated by the $z_\aaa$,  the map $u$ is well-defined and a homomorphism. Now let $\varphi \colon F \to A$ be  the function such that \bc  $u(z)- u'(z) =\varphi(z)$. \ec  The required properties of $\varphi$ are evident from its construction. Since $u$ is a homomorphism  we have $0=u(x) + u(y)- u(x+y)$. So by  (\ref{eqn: hphi}) we have \bc $-h(x,y)= \varphi(x)+ \varphi(y)- \varphi(x+y)$. \ec This yields (\ref{hhh}).
 \end{proof}

 If $F$ is a free group,  $R \le F$, and $A$ is any group, by $\Hom(F| R,A)$ one denotes the subgroup of $\Hom(R,A)$ consisting of those homomorphisms that can be extended to $F$.  The computable setting is as above, with  in addition $R$ being  computable. 
 \begin{thm}[Theorem 10.1 in~\cite{Eilenberg.MacLane:42}]  \label{thm: EMcLane 10.1}
 Suppose that $H= F/R$ is a factor group of a free group $F$. Let  $g \colon H \times H \to R$ be a cocycle for the extension $0\to R\to F \to H \to 0$. Let $A$ be any group. There is an isomorphism
 \[  \Hom(R,A)/\Hom(F| R,A) \cong \Ext(H,A)  \]
 induced  by the homomorphism \bc $\Phi\colon \Hom(R,A)\to Z(H,A) $  given  by   $\Phi(\theta)= \theta\circ g$. \ec 
Its inverse is induced by a  homomorphism  $\Gamma\colon Z(H,A)\to \Hom(R,A)$. Thus  for $f\in Z(H,A)$  one has $\Phi(\Gamma(f))-f\in B(H,A)$.

  In the computable setting,  both $\Phi$ and $\Gamma$ can be evaluated by an oracle Turing machine  applied to $\theta \in \Hom(R,A)$, resp.\ $f\in Z(H,A)$ as  oracles. Since we can view elements of both sets as names in the sense of Def.\ \ref{df:names} (with $A= \Sigma$), this shows that the isomorphism and its inverse are computable.  
 \end{thm}
 
 \begin{proof} As above, corresponding to $g$ pick a coset   representatives via a   $u\colon H \to F$, namely, $R+u(x)=x$ for any $x\in H= F/R$, and 
 \begin{equation} \label{eqn: ug} u(h)+ u(k)  = u(h+k) +g(h,k) \end{equation} for $h,k \in H$. 
 
 First they verify the conditions of $\Phi$. Clearly $\theta\circ g \in Z(H,A)$. They  verify that \bc  $\theta \in \Hom(F| R,A)$ iff $\theta \circ g \in B(H,A)$, i.e.\ is a coboundary. \ec So $\Phi $ induces a well-defined embedding. 
 
 Next they show how to define $\Gamma$.  
 Write $[w]$ for the coset $R+ w \in H$. For $f\in Z(H,A)$ let \bc $f'\in Z(F,A)$ be given by $f'(x,y) = f([x], [y])$. \ec Let $\varphi\colon F \to A$ be the function for $h=f'$ given by (\ref{hhh}). Thus 
 \begin{equation} \label{eqn: phi} \varphi(x+y)  =   \varphi(x) + \varphi(y) +f'(x,y). \end{equation} If $x,y\in R$ then $f'(x,y)=0$ as $f$ is normalized, so $\theta = \varphi |R\in \Hom(R,A)$.  Define \bc $\Gamma(f) = \varphi | R$. \ec    Note that  $\Phi(\Gamma(f) ) = \varphi \circ g $. So to show $\Gamma$ induces the inverse,  the following   suffices:
 
 \begin{claim} For normalized $f\in Z(H,A)$ one has  \bc  $f- (\varphi \circ g ) \in B(H,A)$ via $\varphi \circ u$. \ec \end{claim}
 
 To see this one applies $\varphi$ to the right side $u(h+k)+g(h,k)$ of (\ref{eqn: ug}). According to  
 (\ref{eqn: phi}), this gives  
 \bc  $\varphi(u(h+k))+\varphi(g(h,k)) + f'(u(h+k), g(h,k))$. \ec
 Since $g(h,k)\in R$ and $f$ is normalized, the third term vanishes. So we have
 \[\varphi(u(h))+\varphi(u(k)) = \varphi(u(h+k))+\varphi(g(h,k)).\]
 
 We can now calculate
 
\begin{eqnarray*} f(h,k)= f'(u(h), u(k))&=& \varphi(u(h)+u(k))- \varphi(u(h))- \varphi(u(k)) \\
							&=& \varphi(u(h+k))- \varphi(u(h))- \varphi(u(k)) +\varphi(g(h,k)) \end{eqnarray*}						
where we have first  used (\ref{eqn: phi}) and then the foregoing equation. This verifies the claim.

In the computable setting, it is clear that $\Phi$ can be  evaluated by an oracle Turing machine. To see this for  $\Gamma$, note that the OTM has access to elements of $F/R= H$  because $R$ is computable.  Recall that for $y\in F$ we write $[y]= R+y$.  Fix a computable list of  generators $\seq {z_\aaa}$ as in the statement of the theorem. Each $ w$ in  $F$ can be written as $w=\sum_{i=1}^n x_i$, where $x_i$ is a generator or its inverse. 
\begin{claim}   Given $f\in Z(H,A) $, then $\Gamma(f)=\theta$ is explicitly given as follows:  For an element of $R$ expressed as   $\sum_{i=1}^n x_i$, let  $I$ be the set of indices such that $x_i$ is the inverse of a generator. Then
\begin{equation} \label{eqn:theta f} \theta (\sum_i x_i)=  -\sum_{i \in I} f([x_i],-[x_i])  + \sum_{k=1}^{n-1} f(\sum_{i=1}^k [x_i], [x_{k+1}]). \end{equation}
 \end{claim}
 To verify the claim, note that for $y_i \in F$ an induction based on (\ref{eqn: phi}) shows that 
 $\varphi(\sum_i y_i)=\sum \varphi(y_i)+ \sum_{k=1}^{n-1} f'(\sum_{i=1}^k y_i, y_{k+1})$. Now use that $\varphi(y)=0$ if $y$ is a generator, and $\varphi(y)=-f'(y,-y)$ if $y$ is the negative of a generator; this uses that $0= \varphi(0)= \varphi (y)+ \varphi(-y)+ f'(y,-y)$.  
 
 It is clear that the OTM on input $w\in R$ can find a representation $w = \sum x_i$, and then perform the calculation in (\ref{eqn:theta f}) using oracle calls to $f$. The result will be independent of the particular representation, e.g.\ the order of the $x_i$.
 \end{proof}
 \begin{remark} In fact, $\Phi$ and $\Gamma$ are total Turing functionals, i.e. tt-reductions. \end{remark}
 \begin{example} As one application Eilenberg and MacLane re-verify the known fact  that for $H= C_m$ (the cyclic group of size $m$), one has $A/mA \cong \Ext(H,A)   $. This is clear because one can take $F=\ZZ$, $R=m\ZZ$, so that $\Hom(R,A)$ and $\Hom (F|R,A)$ can be canonically identified with $A$ and $mA$, respectively. 
 
As above let $[k]= m\ZZ +k$.   In the setting of Theorem~\ref{thm: EMcLane 10.1}, choose $k=u([k])$ as the coset  representative of $[k]$, where $k< m$. The corresponding  cocycle $g\colon C_m \times C_m \to R$ in the Theorem is then  given by \bc $g([k], [l])= k+l - (k+l \mod m)$, \ec  which equals  $ 0$ if $k+l < m$, and $m$ if $k+l \ge m$.   

Let us consider  the case $A=\ZZ$. Note that  $\theta \in \Hom(m\ZZ, \ZZ) $ is determined  by its single value $a_\theta=\theta(m)$, in the sense that $\theta(x)=a_\theta x/m$ for each $x\in m\ZZ$. 
Then $\theta \in \Hom(m\ZZ |\ZZ, \ZZ)$ iff $m | a_\theta$. As this means $\Phi(\theta) \in B(C_m, \ZZ)$, we can assume that $0\le a_\theta < m$ without changing the corresponding extension.

Firstly we have $\Phi(\theta) = \frac{a_\theta} m g$. Secondly    given $f \in Z(C_m,\ZZ)$,  we determine 
 $\theta = \Gamma(f)$ as follows.   With $1$ as the generator of the free group $F= \ZZ$,    writing $m$ as the sum of $m$ many $1$'s, by (\ref{eqn:theta f}) we have 
 \[a_\theta=\theta(m)= \sum_{k=1}^{m-1}f([k],[1]).\] 
 
 If $f=\frac{a_\theta} m g$ then this yields $a_\theta$ since $\sum_{k=1}^{m-1}g([k],[1])=m$. Thus, in this simple case of $\Ext(C_m, \ZZ)$,  one can ensure that  $\Gamma(\Phi(\theta))= \theta$. 
 \end{example}

\subsection{Description of $\Ext(C_{p^\infty}, \ZZ)$}

  Fix a prime $p$ and let \bc $H=C_{p^\infty}= \ZZ[1/p]/\ZZ$. \ec  Let $\pi\colon \ZZ[1/p] \to H$ be the projection given by $\pi(w) = \ZZ  + w$.
 Let  $x_{n}:=\pi \left( \frac{1}{p^{n}}\right) $ for $n\in \NN $.

We begin with a presentation of $H$. Let $F$ be the free abelian group with $\mathbb{Z}$-basis $\left(
e_{n}\right) _{n\in \NN }$.
Let $R$  be the subgroup of $F$ freely generated by $e_0$ together with $%
e_{k-1}-pe_{k}$ for $k\geq 1$. Then we have a short exact sequence $%
0 \rightarrow R\rightarrow F\rightarrow H\rightarrow 0$ where the map $F\rightarrow H$ is defined by  $%
e_{n}\mapsto x_{n}$ for $n\in \NN $.

 Define $u \colon H \to F$ by \bc $u(0)= 0$ and  $u \left(
ix_{n}\right) =ie_{n}$ \ec  for every $n\ge 1$ and natural $i < p^n$ not divisible by $p$.
This is a right inverse for the quotent map $F\rightarrow H $.

Since $R$ is computable as as subset of $F$,  by Theorem~\ref{thm: EMcLane 10.1} above there is a computable  isomorphism of abelian groups%
\begin{equation} \label{eqn:CII}
\frac{\mathrm{Hom}\left( R,\mathbb{Z}\right) }{\mathrm{Hom}\left( F|R,%
\mathbb{Z}\right) }\cong \mathrm{Ext}\left( H,\mathbb{Z}\right).
\end{equation}%
%

We now describe a computable  isomorphism \[\mathbb{Z}_{p}\rightarrow \frac{%
\mathrm{Hom}\left( R,\mathbb{Z}\right) }{\mathrm{Hom}\left( F|R,\mathbb{Z}%
\right) }.\]
As already mentioned in Example~\ref{ex:Zp}, we use functions in $\aaa \colon \NN^+ \to \ZZ$ as names for $a=\sum_{n}\aaa(n)p^{n-1}\in \ZZ_p$. Write $a_n= \aaa(n)$. 
\begin{definition}  \label{dfn:phia}
Given a name    $\aaa \colon \NN \to \ZZ$, let 
  $\varphi _{\aaa}\in \Hom(R, \mathbb{Z})$ be defined
by \bc $\varphi _{\aaa}(e_0)= 0$ and $\varphi _{\aaa}\left(
e_{r-1}-pe_{r}\right) =a_{r}$ for $r\geq 1$. \ec  \end{definition}
  Write $\aaa\uhr {k,n}$ for $p^{-k}  \sum_{s=k+1}^n a_s p^{s-1}$.   This is the integer given by the block of the sequence $\aaa$ from $k$ to $n-1$.
 \begin{claim} \label{cl: phiaaa descr} For $0 \le k< n$ one has 
 \[\phi_\aaa(e_k- p^{n-k}e_n)= \sum_{s=k+1}^{n} a_{s}p^{s-k-1} =\aaa\uhr {k,n}.\] \end{claim}
 In particular, if $\aaa \in \ZZ_p$, letting $k=0$, we have that  $\phi_\aaa(-p^n e_n)$ equals $\aaa \uhr n$. To verify this  equation  fix $k$. Use    induction on $n$. 
 
 If $n=k+1$ the equation  reduces to  the defining conditions for $\phi_\aaa$ (including the case that $k=0$).  Suppose now  the equation holds for $n$. Then 
 
 \begin{eqnarray*}   \phi_\aaa(e_k - p^{n+1-k}e_{n+1})&=&  \phi_\aaa( e_k- p^{n-k} e_n  + p^{n-k} (e_n- p e_{n+1})) \\
 									&=& \sum_{i=k+1}^{n} a_{i}p^{i-k-1} + p^{n-k} a_{n+1} \\
 &=& \sum_{i=k+1}^{n+1} a_{i}p^{i-k-1}. \end{eqnarray*}

\begin{claim} Let $\aaa \colon \NN^+ \to \ZZ$ be a name, and write  $a_n=  \aaa(n)$. Then 

\bc  $a:=\sum_{n\in \NN}a_{n}p^{n-1} =0$ in $\ZZ_p $ if and only if $\varphi _{\aaa}\in \mathrm{Hom}\left( F|R,%
\mathbb{Z}\right) $.  \ec \end{claim}

\lapf  
By hypothesis,  $\varphi _{\aaa}$ lifts to a homomorphism $\psi : F\rightarrow \mathbb{Z}$.
Define $b_{n}:=\psi \left( e_{n}\right) $ for $n\in \NN $. 
Thus, \bc $a_{1}=-pb_{1}$ and $a_{k}-b_{k-1} = -pb_{k}$ for each $k \ge 2$. \ec

We prove by induction on $k\ge 2$  the equation \bc $a=\sum_{n\geq k}p^{n-1}a_{n}-p^k b_{k-1} $. \ec This suffices to show $a=0$ because it implies that $p^k |a$ for each $k$.

For $k=2$, we have $a = \sum_{n\geq 2}p^{n}a_{n} - pb_1$ since $a_{1}=-pb_{1}$. Suppose now that the equation holds for $k$. Then 
\begin{eqnarray*} a =\sum_{n\geq k}p^{n}a_{n}-p^k b_{k-1}&=&\sum_{n\geq k+1}p^{n}a_{n}  + p^k(a_k-b_{k-1} ) \\ &=&\sum_{n\geq k+1}p^{n}a_{n}-  p^{k+1} b_{k},  \end{eqnarray*}
 as required.

\rapf Suppose $a=0$. We define recursively $%
b_{1},b_{2},\ldots $ such that $-pb_{1}=a_{1}$ and $b_{n-1}-pb_{n}=a_{n}$
for $n\geq 2$, and then define a homomorphism $\psi :F\rightarrow \mathbb{Z}$
extending $\varphi _{\aaa}$ by setting $\psi \left( e_{n}\right) =b_{n}$ for $%
n\in \NN $.

Since $a=0$ we have  $%
a_{1}\equiv 0\mathrm{\ \mathrm{mod}}\ p$ and hence $a_{1}=-pb_{1}$ for some $b_1$. 

Then
\begin{eqnarray*}
a &=&\sum_{n\geq 2}p^{n}a_{n}-pb_{1} \\
&=&\sum_{n\geq 3}p^{n}a_{n}+p\left( a_{2}-b_{1}\right) 
\end{eqnarray*}%
Since $p^2 \mid a$ in $\ZZ_p$, we have  $p \mid a_{2}-b_{1} $ and hence $%
a_{2}-b_{1}=-pb_{2}$ for some $b_{2}$. Thus,%
\begin{equation*}
a=\sum_{n\geq 4}p^{n}a_{n}+p^{2}\left( a_{3}-b_{2}\right) .
\end{equation*}%
Since  $p^3 \mid  a$ this implies   $%
a_{3}-b_{2}=-pb_{3}$ for some $b_{3}$. One then proceeds in this fashion.

This verifies the claim.
 By the claim, the coset $%
\varphi _{\aaa}+\mathrm{Hom}\left( F|R,\mathbb{Z}\right) $ only depends on $a$
and not on the particular name $\aaa$ of $a$.
This implies that the function $\mathbb{Z}_{p}\rightarrow \frac{\mathrm{Hom}%
\left( R,\mathbb{Z}\right) }{\mathrm{Hom}\left( F|R,\mathbb{Z}\right) }$ is
a group isomorphism. In the computable setting, it is clear from Def.\ \ref{dfn:phia} that an OTM  can compute $\phi_\aaa$ from $\aaa$ and conversely. So the isomorphism is computable in the sense of Def.\ \ref{df:names}.

\bigskip

Thus, we have (computable) isomorphisms%
\begin{equation*}
\mathbb{Z}_{p}\rightarrow \frac{\mathrm{Hom}\left( R,\mathbb{Z}\right) }{%
\mathrm{Hom}\left( F|R,\mathbb{Z}\right) }\rightarrow \mathrm{Ext}\left( 
\mathbb{Z}\left( p^{\infty }\right) ,\mathbb{Z}\right)
\end{equation*}%
as required. 

Next we will explicitly describe a cocycle $c_a$ corresponding to $a \in \ZZ_p$. We fix  a name $\aaa$, write  $a_n$ for $\aaa(n)$. Assume    that $0 \le a_n< p$ for each $n$. 

Recall that  $u :H\rightarrow F$ is the function defined by setting $u(0)= 0$ and  $u \left(
ix_{n}\right) =ie_{n}$ for every $n\ge 1$ and $0\leq i<p^{n}$ not
divisible by $p$. Also recall that in the first step  of constructing the isomorphism above, $\aaa$ is mapped to $\varphi _{\aaa} $   defined in~\ref{dfn:phia}. In the second step $%
\varphi _{\aaa}$ is mapped to the cocycle $c_{\aaa}$ on $\mathbb{Z}\left(
p^{\infty }\right) $ with coefficients in $\mathbb{Z}$ defined by  \[%
c_{\aaa}\left( x,y\right) =\varphi _{\aaa}\left( u \left( x\right) +u \left(
y\right) -u \left( x+y\right) \right) \] for $x,y\in \mathbb{Z}\left(
p^{\infty }\right)$.

We first give an expression for $ u \left( x\right) +u \left(
y\right) -u \left( x+y\right) $ where $x,y \in H   - \{ 0\}$. Suppressing  the projection map $\pi \colon \ZZ[1/p]\to \ZZ(p^\infty)$, we can write $x=i/p^{n}$ and $y=j/p^{k}$ where $i<p^{n}$ and $j<p^{k}$ are natural numbers
not divisible by $p$. 
  Thus, we have that%
\begin{equation*}
u \left( x\right) =ie_{n}, 
u \left( y\right) =je_{k}
\end{equation*}%
and, in case  say, $1 \le k\leq n$, we  have $x+y= ({i+jp^{n-k}\mod p^n}){p^{-n}}$.

\n
\emph{Case  $k<n$:}  then $i+jp^{n-k}\mod p^{n}$ is not divisible by $p$.

\n  If  $i+jp^{n-k} < p^n$ then $u \left( x+y\right) =\left( i+jp^{n-k} \right) e_{n}$, so  %
\begin{equation*}
u(x) + u(y) - u \left( x+y\right) = j(e_n  -p^{n-k} e_{n}).
\end{equation*}%
If  $i+jp^{n-k} \ge p^n$ then  $u \left( x+y\right) = ( i+jp^{n-k} -p^n) e_{n}$, so  %
\begin{equation*}
u(x) + u(y) - u \left( x+y\right) = j(e_n  -p^{n-k} e_{n}) + p^n e_n
\end{equation*}%

\n \emph{Case $k=n$:} then $i+j\equiv p^{n-k' }v\mathrm{\ \mathrm{mod}}\ p^{n}$ for some $%
v<p^{k' }$ not divisible by~$p$ and hence $x+y=\frac{v}{p^{k' }}$, $u \left( x+y\right) =ve_{k' }$, and %
\begin{equation*}
u(x) + u(y) - u \left( x+y\right)  =(i+j)e_n -ve_{k' }
\end{equation*}%
Secondly, using  Claim~\ref{cl: phiaaa descr} we evaluate $c_\aaa(x,y)$ by applying  $\phi_\aaa$:  

\n  Let $d=1 $ if $ i/p^{n} + j/p^{k} \ge 1 $ in $\ZZ [1/p]$, and $d=0$ otherwise.

\n Write $\aaa\uhr {k,n}$ for $p^{-k} \sum_{s=k+1}^n a_s p^{s-1}$ as before.   

\bi \item if $ k<n$ :
\begin{eqnarray*} 
 {c}_{\aaa}\left( i/p^{n},j/p^{k}\right) &=&\varphi _{\aaa}(je_{k}-jp^{n-k}e_{n} - d p^n e_n) \\ &=& j\aaa\uhr{k,n} + d \aaa \uhr {0,n}
 \end{eqnarray*}
 
%
 \item  if $ 
n=k,i+j\equiv p^{n-k' }v\mathrm{\ \mathrm{mod}}\ p^{n}\text{ for }%
v<p^{k' }\text{ not divisible by }p%
  $  
  
\begin{eqnarray*}
 {c}_{\aaa}\left( i/p^{n},j/p^{k}\right) &=& \varphi _{\aaa}((p^{n-k' }v + d p^n)  e_{n} -ve_{k' }) \\ &=& v \varphi_\aaa( p^{n-k'} e_n - e_{k'}) +\phi_\aaa(d p^n  e_{n} )  )\\ &=& -v \aaa\uhr {k',n}-d \aaa \uhr {0,n} \end{eqnarray*}

\ei
 
 \subsection{Description of $\Ext(\ZZ[1/p], \ZZ)$}

 \begin{lemma} 
The  following homomorphism 
\bc $
\mathrm{Ext}\left( \mathbb{Z}\left( p^{\infty }\right) ,\mathbb{Z}\right)
\rightarrow \mathrm{Ext}\left( \mathbb{Z}[1/p],\mathbb{Z}\right)
$ \ec   is onto and has kernel $\ZZ$:   a cocycle 
$c\in Z(\ZZ(p^\infty), \ZZ))$ is mapped to the cocycle $\widetilde{c}%
$ on $\mathbb{Z}[1/p]$ with coefficients in $\mathbb{Z}$ defined by 
 $
\widetilde{c}\left( x,y\right) =c\left( \pi \left( x\right) ,\pi
\left( y\right) \right)  
$.
\end{lemma}
\begin{proof} One uses the long exact sequences of Cartan and Eilenberg; see e.g.\  the first statement in Fuchs~\cite[Ch.\ 7, Th.\ 2.3]{Fuchs:15}. Starting from the exact sequence $0\to \ZZ \to \ZZ[1/p] \to \ZZ(p^\infty)\to 0$, one obtains an exact sequence

  $\Hom (\ZZ[1/p], \ZZ) \to \Hom(\ZZ, \ZZ) \to $ 
  
  \hfill $\Ext( \ZZ(p^\infty), \ZZ) \to \Ext(\ZZ[1/p], \ZZ) \to \Ext(\ZZ, \ZZ) \to 0$.  
  
  Since 
  $\Hom (\ZZ[1/p], \ZZ) =0$,  $ \Hom(\ZZ, \ZZ) \cong  \ZZ$, and $\Ext(\ZZ,\ZZ)=0$, the third connecting map is   an epimorphism with kernel  isomorphic to $\ZZ$. The proof in Fuchs shows that this  map is induced by combining the projection an cocycles as required.
\end{proof}
This gives an epimorphism with kernel $\ZZ$%
\begin{equation*}
\mathbb{Z}_{p}\rightarrow \mathrm{Ext}\left( \mathbb{Z}[1/p],\mathbb{Z}.%
\right)
\end{equation*}%
Question: describe, up to sign,  the   $a\in \ZZ_p$ that   generates  the kernel.

 \subsection{Describing all the rank-2 subgroups of $\mathbb{Z}[1/p]^{2}$}

Suppose   $K\subseteq \mathbb{Z}[1/p]^{2}$  has rank 2.
Let $A\subseteq K$ be a $p$-basic subgroup:  since $K$ is torsion free this means it is a direct sum of  infinite cyclic groups, $A$ is $p$-pure in $K$ (i.e., $pK \cap A   = pA$), and $K/A$ is $p$-divisible (i.e., $K/A= p(K/A)$).
 
If $A=0$ then $K$ is $p$-divisible. 
If $A=K$ then $K$ is free.
Otherwise,  $A$ has rank $1$,  so  $A\cong \mathbb{Z}$ and $K/A$ has rank $1$ and is $p
$-divisible; thus $K/A\cong \mathbb{Z}[1/p]$.
In this case, $K$ is an extension of $\mathbb{Z}[1/p]$ by $\mathbb{Z}$.

Thus,   $K$ can be seen to have  domain $\mathbb{Z}\times \mathbb{Z}[1/p]$ with operation
defined by, for some $a\in \mathbb{Z}_{p}$,%
\begin{equation*}
\left( x,i/p^{n}\right) +\left( y,j/p^{m}\right) =\left( x+y+\widetilde{c}%
_{a}\left( i/p^{n},j/p^{m}\right) ,\frac{i}{p^{n}}+\frac{p^{n-m}j}{p^{n}}\right) 
\end{equation*}%
in case $m\le n$, and symmetrically  if $m \ge n$.
 
%
%
%
%

    \section{Nies: questions on groups and logic} We provide some open questions connect group theory and logic.  They have in part been discussed with Segal.

\subsection{Profinite groups}

1.   
  Are  the pro-$p$ completion of  the $F_n$, for $n \ge 2$, finitely axiomatizable   (FA) in the pro-$p$ groups/ all profinite groups?
 
 
 
2.  Are simple Lie-algebras over $\ZZ_p$ FA in the class of Lie algebras over $\ZZ_p$? This would be  a Lie analog to the result on powerful pro-$p$ groups. Same
for $sl_{2}(Z_{p})$ and such algebras.

%

\subsection{Pseudofinite groups}

$C_2^{(\omega)}$ is  a pseudofinite group   with a solvable word problem. Its theory  is axiomatised  by saying the group is infinite, and of exponent 2. Each finite set of axioms holds in some finite quotient.

1. Is there an infinite f.g.\ pseudofinite group with solvable WP? If not then f.p. pseudofinite groups, being r.f., cannot be infinite.

2. What are the $\omega$-categorical extensions of the theory of finite groups? $\mathrm{Th} (C_2^{(\omega)})$ is an example.  I tried for a bit the free  group of infinite rank  in the variety of  nilpotent-2 groups of exponent 3. 

3. Which groups are FA in the class of countable pseudofinite groups? Again $C_2^{(\omega)}$ is an example.


 \subsection{QFA groups} Recall that a f.g.\ infinite group $G$  is  QFA is there is a f.o.\ sentence such that $G$  is up to isomorphism its  only model among the f.g.\ groups. 
 
 1. Long standing  question: is each QFA group prime? (I.e., is each $n$-orbit, $n \ge 1$ definable?)
 
 This question also makes sense for profinite groups. One could define a concept of primeness by saying that for each $n$ the tuples with definable orbits are dense. 
 
 In the pro-$p$ case, a better way would be that each orbit is definable in the language with p-adic power operations. 
 $\ZZ_p$ might be an example.

 2. Can a QFA group be torsion?  (This appears to be  related to the questions whether there is a f.p.\ infinite torsion group.) How about infinite Burnside groups, can they be QFA?
 
%
  \part{Metric spaces and descriptive set theory}
   \newcommand{\Op}{\text{\it Op}}
\newcommand{\Inv}{\text{\it Inv}}

\section{Melnikov: Polish metric spaces and t.d.l.c.\ groups}

\subsection{Definitions from computable metric space theory}
\begin{definition}
A Polish space $M$ is \emph{computable Polish} or \emph{computably metrized} if there is a compatible, complete metric $d$ and a countable sequence of \emph{special points} $(x_i)$ dense in $M$ such that, on input $i, j, n$, we can compute a rational number $r$ such that $|r - d(x_i, x_j)| <2^{-n}$. 
\end{definition}
We allow the possibility that $d(x_i, x_j) =0$. However, it is not difficult at all to exclude repetitions from the dense set if necessary; we will not need this in our proofs.

\

A \emph{basic open ball} is an open ball having a rational radius and centred in a special point. Let $ X$ be a computable Polish space, and $(B_i)$ is the effective list of all its basic open balls, perhaps with repetition. (We also sometimes write $B_r(x)$ for the open ball having radius $r$ and centred in $x$: $B_r(x) = \{y: d(x,y)<r\}$.)

\begin{definition}\label{def:point} 
We call 
\[
N^x = \{i: x \in B_{i}\}
\] 
\emph{the name of $x$} (in $ X$).  
\end{definition}

We can also use basic open balls to produce names of open sets, as follows. 
A \emph{name} of an open set $U$ in a computable topological space $X$ is a set  $W \subseteq \mathbb{N}$ such that $U = \bigcup_{i \in W} B_i$, where $B_i$ stands for the $i$-th basic open set in the basis of $X$.
 If  an open $U$ has a c.e.~name, then we say that $U$ is \emph{effectively open}.

\begin{definition} \label{def:cont}
A function $f\colon X\to Y$  between two computably metrized Polish spaces is {\em effectively continuous}  if there is a c.e.\ family $F \subseteq \mathcal{P}( X) \times \mathcal{P}( Y)$ of pairs of (indices of) basic open sets in  such that:
\begin{itemize}
\item[(C1):] for every $(U,V)\in F$, $f(U)\subseteq V$;
\item[(C2):]  for every $x \in  X$ and basic open $E \ni f(x)$ in $ Y$  there exists a basic open $D \ni x$ in $ X$ such that $(D, E) \in F$. 
\end{itemize}
\end{definition}

Note that a function is continuous if and only if it is effectively continuous relative to some oracle. 

Recall that an enumeration operator $\Phi$  is given by a c.e.\ set $S$ of pairs of natural numbers. On input $Y$, one has $\Phi^Y= \{ n \colon \ex r \,[ \la r, n \ra \in S \lland D_r \sub Y]\} $.   Informally, $\Phi$   only uses ``positive" information from  $Y$, and hence turns enumarations into enumerations. 
The lemma below is well-known.

\begin{lemma}\label{le: computable map}
Let $f\colon X\to Y$ be a function between computable Polish spaces.
The following are equivalent:
\begin{enumerate}

\item $f$ is effectively continuous.
\item There is an enumeration operator $\Phi$ that on input a name of an open set $Y$ (in $Y$), lists a name of $f^{-1}(Y)$ (in $ X$).
\item There is an enumeration operator $\Psi$, that given the name of $x \in X$, enumerates the name of $f(x)$ in $ Y$.     \label{part: Phi}
\item There exists a uniformly effective procedure that on input  a fast Cauchy name of $x \in M$  lists a fast Cauchy name of $f(x)$ (note that the Cauchy names need not be computable).
\end{enumerate}
\end{lemma}

\begin{definition}\label{def:open} A function $f\colon X\to Y$ is \emph{effectively open} if there is a c.e.\ family $F$ of pairs of basic open sets such that 
\begin{itemize}
\item[(O1):] for every $(U,V)\in F$, $f(U)\supseteq V$;
\item[(O2):]  for every $x \in  X$ and any basic open $E \ni x$ there exists a basic open $D \ni f(x)$  such that $(E,D) \in F$. 
\end{itemize}
\end{definition}

The lemma below is elementary.

\begin{lemma}\label{le: eff open}\cite{Melnikov.Montalban:18}  
Lef $f\colon X\to Y$ be a function between computable Polish spaces.
The following are equivalent:
\begin{enumerate}
\item $f$ is effectively open.
\item There is an enumeration operator that given a name of an open set $A$ in $X$, outputs a name of the open set $f(A)$ in $Y$.

\end{enumerate}
\end{lemma}

In particular, if $f$ is a computable and is a homeomorphism, then it is is effectively open if, and only if, $f^{-1}$ is computable. In this case we say that $f$ is a computable homeomorphism.

We also say that two computable metrizations on the same Polish space are effectively compatible if the identity map on the space  is a bi-computable homeomorphism when viewed as a map from the first metrization to the second metrization under consideration.

\begin{definition}\label{def:effcomp}
A compact computable Polish  space is \emph{effectively compact} if there is a (partial) Turing functional that given a countable cover of the space outputs it finite subcover (and is undefined otherwise).
\end{definition}

This is equivalent to saying that, for every $n$, we can uniformly produce at least one finite open $2^{-n}$-cover of the space by basic open balls; see \cite[Remark 2.5]{Lupini.etal:21}.  The following elementary fact is well-known:

\begin{lemma}
A computable image of an effectively compact space is itself effectively compact.
\end{lemma}

\begin{proof} List basic open balls in the image until their preimages finally cover the domain of the computable map. We will also use that the inverse of bijective computable map $f:X \rightarrow Y$, where $X$ and $Y$ are effectively compact, is also computable. Also, it is well-known that both the supremum and the infinum of a computable function $f: X \rightarrow \mathbb{R}$ is computable provided that $X$ is effectively compact, and this is uniform.  \end{proof} 

\subsection{Splitting a space into clopen components}
Given a computable metric space, one says  that a basic open ball $B_r(x)$ is \emph{formally contained} in a basic open ball $B_q(y)$ if $d(x,y) +q < r$; this is $\Sigma^0_1$ in the given parameters. Two balls are \emph{formally disjoint} if  the distance between their centres exceeds the sum of their radii; this is also $\Sigma^0_1$.

\begin{lemma}\label{lem:clopen}
Suppose  $M$ is effectively compact. Then there is a computable enumeration of all clopen splits of $M$ (perhaps, with repetition).
\end{lemma}

\begin{proof} 
Suppose $M = X \sqcup Y$ is a clopen split, and let $\delta$ be the infimum-distance 
between these compact open sets
$$\delta = \inf_{(x,y) \in X \times Y} d(x,y).$$
(Since $X\times Y$ is compact and $d$ is continuous, it attains its infimum at some pair $(x_0, y_0)$. In particular, $\delta>0$.)

 Suppose $0< \epsilon< \delta/4$. Then every finite $\epsilon$-cover will consist of two formally disjoint subsets of basic open balls. Indeed, every ball covering a point in $X$ cannot contain a point in $Y$, and every ball covering a point in $Y$ cannot contain a point in $X$. If a basic open $B$ has its centre in $X$ and $D$ has its centre in $Y$, then the distance between their centres is at least $\delta$, while the sum of their radii is at most $\delta/2< \delta$, making them  {formally disjoint}. On the other hand, if a finite open cover of $M$ consists of two formally disjoint subcovers,
then these subcovers induce a split of $M$ into clopen components.
Since  being formally disjoint is a c.e.~property, we can effectively list all such clopen splits.
 \end{proof}

The fact below (though not as stated) is due to M.~Hoyrup, T.~Kihara, and V.~Selivanov~\cite{Hoyrup.etal:20}.
It can also be recovered from  Brattka, le Roux, Miller, Pauly \cite{Brattka.etal:19}.

\begin{theorem}\label{thm:stuff}
Given an effectively compact Stone space $M$, one can effectively determine a computable, computably branching tree $T$ without dead ends and a computable homeomorphism $f: M \rightarrow [T]$.
\end{theorem}

\begin{proof}
Suppose $X$ and $\neg X = M\setminus X$ are  clopen components represented as  finite unions of formally disjoint basic open balls, as in the proof of the lemma above. 
Given a special point $x$ in $M$, we can use these finite open names to wait and see whether $x$ in $X$ or $x$ in $\neg X$. This makes both $X$ and $\neg X$ computable closed subsets of the effectively compact $M$, and thus then can be viewed as effectively compact spaces. In particular, their diameters are computable reals. (Note this is uniform.)

\begin{lemma}
Given two (finite open names for) clopen sets $X$ and $Y$, as well as their complements $\neg X$ and $\neg Y$,
we can additionally decide whether $X \cap Y$ is empty, and if it is not empty, then output a finite open name of it and its complement.
\end{lemma}

\begin{proof} 
Search for an  $\epsilon$-cover of the space $M$, where $\epsilon$ is so small that every ball in the new cover is formally contained in some ball of each of the two covers that we fixed above (the first  for $X$ and $\neg X$, and the second for $Y$ and $\neg Y$). Such a cover must exist.
Then $X \cap Y$ is composed of those balls in the cover that are formally contained in balls corresponding to names of $X$ and of $Y$. If there are no such, then declare the intersection $X \cap Y$ empty.\
\end{proof}

To build the tree $T$, associate the empty string with $M$. Suppose $\sigma \in T$ of length $i$ has been defined, and suppose $\sigma$ has been associated with a clopen $X$. Let $X_i \sqcup \neg X_i $ be the $i$th clopen split of $M$ in the effective list of all such splits produced above. If both $X \cap X_i$ and $X \cap \neg X_i$ are non-empty, then create two children of $\sigma$, $\sigma\hat{}0$ and $\sigma\hat{}1$, and associate $X \cap X_i$ with 
$\sigma\hat{}0$ and $X \cap \neg X_i$ with $\sigma\hat{}1$. If only one of the $X \cap X_i$ and $X \cap \neg X_i$ is non-empty, say $X \cap X_i\neq \emptyset$, then create only 
$\sigma\hat{}0$  and associate it with $X \cap X_i$. 

It should be clear that $[T]$ is homeomorphic to $M$. We claim that it is computably homeomorphic to $M$. For that, not that for every $\xi \in T$ and any $n>0$, we can compute (uniformly in $\xi$) an $i$ such that the diameter of the clopen component associated with $\xi \upharpoonright i$ is at most $2^{-n}$.
We identify $[\sigma]$ with the clopen component of $M$ associated with $\sigma \in T$.

 Given a (not necessarily computable) point $x \in M$ and $n$, search for $\sigma \in T$ such that the component of $M$ associated with $\sigma$ has diameter at most $2^{-n}$ and $x \in [\sigma]$.  Output (any point in) $[\sigma]$. This gives a computable name of a surjective computable $f$ (that can be viewed as the identity map) between the effectively compact spaces $M$ and $[T]$. Since both spaces are effectively compact, $f^{-1}$ is also computable.
\end{proof}

\begin{remark}\label{rem:squishy_metric_after_military_service}
We see that, under a careful choice of notation in the end of the proof above, $f$ can be viewed the identity map on~$M$. In other words, the metric induced by $T$ is \emph{effectively compatible} with the original metric on $M$. In particular, any operation defined on $M$ that is computable wrt the old metric will also be computable wrt the new ultrametric induced by $T$.
\end{remark}

We shall use the following observation. Write $A \cong_{comp} B$ is $A$ is computably homeomorphic to $B$.

\begin{remark}\label{rem:stuff}
Suppose $M = C \sqcup D$ is effectively compact, where $C$ and $D$ are clopen and effectively compact. Let $T_0$ and $T_1$ be computably finitely branching trees with no dead ends such that $C \cong_{comp} [T_0]$ and $D \cong_{comp} [T_1]$. Define a new tree $T$ by adjoining the successors of the root of $T_1$ to the root of $T_0$. Then $M \cong_{comp} [T]$ in a   uniform way. 
\end{remark}

\subsection{Computable Baire vs computable $\sigma$-compact open.}
 \begin{definition} \label{def:comploccompact} Let $T$  be a computable subtree of $ \NN^*$ without leaves. We say that  $T$ is \emph{computably  locally compact (c.l.c.)}~if 
\begin{enumerate} \item the space $[T]$ is locally compact, 
\item   the set  
   $  \{ \sss\in T \colon [\sss]_T \text{ is compact}\}$ is decidable, and 
   \item the tree of extensions of each string $\sss $  such that $[\sss]_T \text{ is compact}$
     is uniformly computably branching. More formally, 
 there is a computable binary   function $h$ such that,  if   $ [\sss]_T $ is compact  and $\rho \in T$ extends $\sss$, then  $\rho(i) \le H(\sss, i)$ for each $i < |\rho|$.       \end{enumerate}
\end{definition}

\begin{definition}\label{def:main1}  
       Let $G$ be a Polish t.d.l.c.~group. A \emph{computable Baire~presentation} of $G$
    is a topological group $\hat G \cong G$ of the form  $\hat G= ([T], \Op, \Inv)$ such that
  \begin{enumerate} \item  $T$ is computably   locally compact;
  \item  the group operations, $\Op\colon [T] \times [T] \to [T]$ and $\Inv \colon [T] \to [T]$ are computable;

 \end{enumerate}
  \end{definition}

Note that the tree $[T]$ from Def.~\ref{def:main1} induces a computable, complete metric on $G$. It is not difficult to list a computable dense sequence of points in $[T]$. 
This clarifies (2)  above. The exact method of choosing  a computable dense sequence does not change computability of (2), as long as the resulting (computable metrized) presentation is compatible with the computable topology of the tree.

Note that the metric induced by $T$ is an ultrametric. 
There are     other ways to metrise a totally disconnected space;  the most well-known example is the metric on Cantor space $2^{\omega}$ induced from the standard Euclidean metric on $[0,1]$.

Recall that every locally compact Polish space is $\sigma$-compact, i.e., is the union of a countable ascending sequence of its compact subsets.
We say that a locally compact space $M$ is \emph{computably $\sigma$-compact} if 
$$ M = \varinjlim C_i,$$
where $(C_i)_{i \in \omega}$ is a uniformly computable sequence of (uniformly) effectively compact spaces, and the limit is taken with respect to the injective inclusion maps $f_i: C_i \rightarrow C_{i+1}$ which are also uniformly computable.
Note that $ M = \varinjlim C_i$ is a computable Polish space.

\begin{definition} A separable tdlc group $G$ is computably $\sigma$-compact-open
if the underlying space is  computably $\sigma$-compact as witnessed by  $(C_i, f_i)_{ i \in \omega}$, and additionally:

\begin{enumerate}
\item $f_i(C_i)$ is open in $C_{i+1}$;

\item The group operations are computable on $G = \varinjlim C_i$.

\end{enumerate}

\end{definition}

\begin{lemma}
On $G = \varinjlim C_i$,   left and right translations by a point $\xi$ are uniformly computably open relative to $\xi$.
\end{lemma}

\begin{proof}
Recall that $\xi^{-1}$ is computable from $\xi$, and $f_\xi f_{\xi^{-1}} = Id$.
Since the product is computable, the left translation by $\xi^{-1}$ is computable in $\xi$.  This makes the left translation by $\xi$ effectively open. The same holds for the right translation.\end{proof}

The lemma   applies to Baire presentations as well, as a special case. See the first half of the proof of the theorem below. 

\

\begin{theorem} For a separable tdlc group $G$, the following are equivalent:

\begin{enumerate}
\item $G$  is computably $\sigma$-compact-open;

\item $G$ has a computable Baire presentation.
\end{enumerate}
Furthermore, this correspondence is   uniform.
\end{theorem}

\begin{proof}
(2)$\to$(1). We may identify  $G$  with its computable Baire presentation. So $G = [T]$ for an effectively locally compact  tree $T$. Note that $T$ induces a computable, complete metric on $G$. We can also list a computable dense subset $(x_i)_{i \in \omega}$ of $G$. To define $C_i$, fix any compact open subgroup $U$ of $G$. By the choice of $T$, it is an effectively compact space. Then note that, since $U$ is open,  every (left) coset of $U$ has the form   $\xi U$, where $\xi$ is a special point.
Define $C_i = \bigcup_{j\leq i} x_i U$. The inclusion maps   $C_i \to C_{i+1}$ are evidently computable, and each $C_i$ is an effectively compact open set in a uniform way. The group operations are computable on $[T]$ and remain computable on the direct limit of $C_i$, since   one can view $[T]$ as the result of taking the direct limit.

\

(1)$\to$ (2). Suppose $G$ is computably $\sigma$-compact-open. Observe that each $C_i$ is an effectively compact Stone space.  By Theorem~\ref{thm:stuff}, there is a   computable procedure that, given an effectively compact Stone space $C_i$, outputs a computably branching, computable tree $T_i$ with no dead ends such that $C_i \cong_{comp} [T_i]$. 

Observe also that, from  $i$, one can compute a code (i.e., as a finite union of basic open balls) for $f_i(C_i)$  as  a clopen subset of $C_{i+1}$ ; this is because it is a computable homeomorphic image of an effectively compact space inside an effectively compact space. For its    complement  in  $C_{i+1}$   one can also compute such a code: as in  the proof of Lemma~\ref{lem:clopen}, list open finite covers of $C_i$ and search for a cover of the whole $C_{i+1}$ that is composed of the cover of $C_i$ and finitely many balls formally disjoint from it. We are therefore in the position to apply Remark~\ref{rem:stuff}.

We see that the tree $T_{i}$ that can be uniformly  produced for $C_{i}$ (by Theorem~\ref{thm:stuff}) can be viewed as a subtree of $T_{i+1}$ corresponding to $C_{i+1}$; this is Remark~\ref{rem:stuff}. It follows that we can define $T$ to be $\bigcup_{i \in \omega} T_i$. 
By Remark~\ref{rem:squishy_metric_after_military_service} the original metric on $C_i$ is uniformly effectively compatible with the new ultrametric induced by the tree $T_i$.  It follows that the group operations remain computable and effectively open on $[T]$.
Since $T_i$ are  computably branching uniformly in $i$, $T$ is computably locally compact.\end{proof}

We sketch an application to showing that certain t.d.l.c.\ groups are computable.
\begin{example} Let $G$ be an algebraic group over $\QQ_p$, e.g.\ $SL_n(\QQ_p)$. Then $G$ has a computable Baire presentation. \end{example}

\begin{proof} $G$ is given as the set of matrices in $M_n(\QQ_p)$ with components satisfying finitely many polynomial equations over $\QQ_p$ (such as $\det A =1$). 

Given $i \in \NN$ let $C_i$ be the set of matrices in $G$ such that each entry is in $p^{-i}\ZZ_p$.  With the matrix norm given by the  maximum distance between corresponding matrix components, $C_i$ is effectively compact uniformly in $i$, and  $G = \bigcup_i C_i$. The inclusion maps $f_i$ are trivially computable and have open range. Clearly $G$ is computably $\sss$-compact-open via this union. 
\end{proof}

 \part{FA-presentable structures}
  \section{IMS workshop, problem session on FA-presentable structures}

At the IMS automata and learning theory workshop in September the following open questions were discussed.

\begin{question}[Stephan]
Is the isomorphism problem for FA presentable structures  in the following classes decidable?

1. Successor structures with a unary predicate

2. Abelian groups.

\end{question} 

\begin{question}[Stephan] Is every f.g. semi-automatic  group Cayley automatic? \end{question}

\begin{question}[Andr\'e Nies]
Is each     torsion-free FA presentable group virtually abelian?  \end{question}

Isomorphism of FA presentable equivalence structures is $\Pi_1$ complete as a set  of pairs (Dietrich Kuske, Jiamou Liu and Markus Lohrey, Trans. Amer. Math. Soc. 365 (2013), 5103-5151). 

\begin{question}[Andr\'e Nies] Is isomorphism of FA presentable equivalence structures  $\Pi_1$ complete as an equivalence relation? \end{question}
\begin{question}[Andr\'e Nies] Let $p$ be some prime. Is there an indecomposable FA presentable subgroup of rank $\ge 2$ of $\ZZ[1/p]^k$?  \end{question}

\begin{question}[Dimitry Berdinsky]  Is there a Cayley automatic representation of a nonabelian nilpotent group $\pi: L \rightarrow G$
for which the domain $L$ has polynomial growth (equivalently $L$ is simply starred)? \end{question}

\begin{question}[Murray Elder] 
Is some nilpotent of class $>2 $ group  Cayley automatic? \end{question}

Stephan and others have studied semi-automatic groups: only the inverse and the transvections are automata presentable. They have  produced a properly  3-nilpotent semi-automatic  group.  Cayley automatic implies semi-automatic. They ask whether the converse holds. If so this answers the question in the affirmative.

\begin{question}[Philipp Schlicht] Is there an algorithm to determine  the ordinal type of  a tree automatic  well-order? \end{question}

Schlicht also asked what one can  say about the Morley rank of an FA-presentable structure.  
  
%
%
%

 \part{Set theory}
           
           \section{Yu: Some consequences of Turing determinacy and strong Turing determinacy}
This is  \href{https://arxiv.org/abs/2107.10470}{joint work} with Liuzhen Wu and Yinhe Peng.  \begin{definition}
Let $sTD$, strong Turing determinacy, be that for any subset $A$ of reals ranging Turing degrees cofinally, $A$ has a pointed subset~$T$. 
\end{definition}

We can prove the following results of Woodin.
\begin{theorem}[Woodin]
Assume $ZF+sTD$,
\begin{itemize}
\item Every set of reals is measurable.
\item Every set of reals has Baire property.
\end{itemize}

If additionally  $DC_{\mathbb{R}}$, then every uncountable set of reals has a perfect subset.
\end{theorem}
\begin{proof}[Sketch of proof]
Both (1) and (2) are proved via classical randomness/genericity theory. The handwritten notes can be found in Yu's homepage.

To prove perfect set property ($PSP$), one needs prove that non-existence of Bernstein set implies $PSP$ by applying $DC_{\mathbb{R}}$ (the idea is from Sami \cite{Sami:89}). The proof can also be found in Yu's homepage.
\end{proof}

It is a long standing question whether $ZF+TD$ (or $ZF+AD$) implies $DC_{\mathbb{R}}$.

\begin{definition}
Let $wDC_{\mathbb{R}}$ be the statement that for any binary relation $R\subseteq \mathbb{R}^2$ with the property that $\forall x (\mu(\{y\mid R(x,y)\})>0)$, there is a sequence $\{x_n\}_{n\in \omega}$ so that $\forall n R(x_n,x_{n+1})$.
\end{definition}

\begin{proposition}
$ZF+TD\vdash wDC_{\mathbb{R}}$.
\end{proposition}
\begin{proof}
The proof is based on some results from higher randomness theory. The full details will appear in somewhere else.
\end{proof}

Note that we are also able to prove the consequence if one replaces positive measure with  having Baire property and nonmeager. 

We are also able to prove the following result.

\begin{theorem}\label{theorem: capacity}
Assume $Zf+sTD $, for any set $A$ of reals, there is a $F_{\sigma}$ subset $F\subset A$ so that $F$ and $A$ have the same Hausdorff dimension.
\end{theorem}

The conclusion of Theorem of \ref{theorem: capacity} was also proved by Crone, Fishman and Jackson under the assumption $ZF+AD+DC$. Slaman proved that the conclusion fails for some $\Pi^1_1$ set under the assumption $V=L$. Also the results remain true if Hausdorff dimension is replaced with packing dimension.

Also it can be proved that if $ZFC$ is consistent, the it is consistent with $ZFC$ that the Hausdorff dimension of every $\mathbf{\Sigma}^1_2$-set can be approximated by its closed subsets.

Actually the following result can be proved.
\begin{theorem}
Assume $ZFC+MA+\neg CH$, if $A=\bigcup_{\alpha<\kappa}B_{\alpha}$, where $\kappa<2^{\aleph_0}$, then  $Dim_H(A)=\sup\{Dim_H(B_{\alpha})\mid \alpha<\kappa\}$.
\end{theorem}
           \section{Yu: On the Hausdorff dimension of Hamel bases.}
Jack  Lutz proves the following result
\begin{theorem}[Lutz]\label{theorem: lutz hamel bases}
Assume $ZFC+ CH$, the Hausdorff dimension of Hamel bases   range over $(0,1]$.
\end{theorem}

Then Renrui Qi  and I independently observed that  Theorem \ref{theorem: lutz hamel bases} can be proved within $ZFC$.

\begin{theorem}[Qi and   Yu]
It is a $ZFC$ theorem that the Hausdorff dimension of Hamel bases   range over $[0,1]$.
\end{theorem}
\begin{proof}(very sketchy)
The proof is basically a straightforward diagonization. The point is that for any real $r$, there are two reals $x_0$ and $x_1$ with effective Hausdorff dimension $0$ so that $r=x_0+x_1$. Then we may apply Lutz-Lutz's point-to-set theorem to construct a Hamel base with Hausdorff dimension $s$ for any given $s\in [0,1]$.
\end{proof}
          
   
 \part{Mathematical logic and quantum mechanics}
 
   \section{Nies: The spectral gap problem for  spin chains}
The following is adapted from the first part of a talk of Nies at the  LQCAI conference, \url{https://lqcai.org/},  at Arak University in Iran, July. 
This part of the talk   reviewed   the  result of Cubitt, Perez-Garcia and Wolf~\cite{Cubitt.etal:15},  showing that  the existence of a spectral gap is undecidable for the two-dimensional (2D) case. The  improvement  by Bausch, Cubitt, Lucia and Perez-Garcia~\cite{Bausch.etal:20}  to the one-dimensional  (1D) case was also mentioned.

 The second part of the talk considered  work (much less known) with    Volkher Scholz~\cite{Nies.Scholz:18}, where we   extend  the notion of Martin-L\"of randomness from the setting of  infinite bit sequences  to  the setting of  infinite spin chains of qubits. 
We showed that there is a universal algorithmic test for randomness, and worked towards a characterisation of this randomness notion via incompressibility of the initial segments, similar to the Levin-Schnorr theorem.   

  A    PhD thesis by Tejas Bhojraj at Univ.\ of  Madison~\cite{Bhojraj:21} on the topic appeared in 2021. It is   available at \url{arxiv.org/pdf/2106.14280.pdf}. There are  three corresponding  publications by the same author (J.\ Math.\ Physics, Theor.\ Computer Sc., ENTCS).   He's recently moved to study  at CMI in Chennai.
  
  To find connections between the two approaches to spin chains, one would have to first formulate a version of the Bausch et al.\ results for infinite  chains of qudits.  Hamiltonians have been studied in this infinitary setting, but they are usually not bounded (i.e., continuous) as operators. Also, they are  only defined on a dense Hilbert subspace. For more detail search the  papers  and books by Nachtergaele, Naaijkens~\cite{Naaijkens:17}, or Sims. Also see the work on quantum dynamical systems by Bjelakovic et al.~\cite{Bjelakovic.etal:04}.

\subsection{Background on spin chains} Spin chains were introduced in physics in the 1920s, mainly as a model for  magnetism. A classical  spin chain consists of  $N$ dipoles arranged linearly:
\bc  $\underbrace{\uparrow \ \downarrow \ \downarrow \ \uparrow \ \ldots \ \uparrow}_N $ \ec

The  positions $i=1, \ldots, N$ in a spin chain    are called \emph{sites}. 
 The energy of a state of the system is given by a Hamiltonian.  The 1D  Ising model is due to Lenz (1920), and was ``solved" by his student Ising in his thesis (1925).
For $N$ sites, the Hamiltonian is 
\bc  $H_N= -J \sum_{i=1}^{N-1} \sss_i \sss_{i+1} -h \sum_{j=1}^N \sss_j$,  \ec where  
  \bi \item $J$ is the interaction strength between neighbours, \item$h$ is the strength of the external magnetic field, \item 
    $\sss_i = 1$ for a  $\uparrow$ at site $i$, and  $\sss_i = -1$ for a  $\downarrow$    at site $i$. 
\ei
      Higher-dimensional arrangements of dipoles have also be studied, in particular  square lattices, which can be pictured as follows: 
  \bc  \scalebox{.6}{\includegraphics{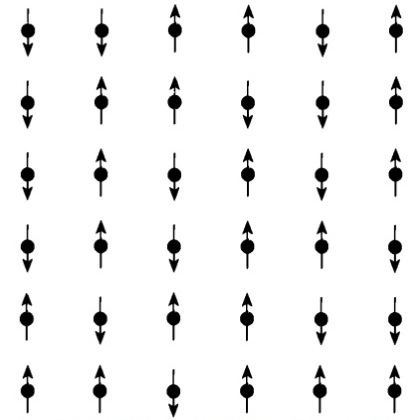}} \ec


Heisenberg (1928)  introduced a better model of magnetism that works in    the quantum setting. He used 
  $n$-chains as before, where each  site now contains a spin $1/2$ particle (a Fermion, e.g., electron).  
The state of such a system is a  unit vector in $(\mathbb C^2)^{\otimes n}$.  
  Spins are modelled  in $x,y,z$ directions, corresponding to observables given by the Pauli matrices $\sss^x, \sss^y, \sss^z $ (certain well-known $2 \times 2$ matrices over $\mathbb C$ which together with $\mathbb I_2$ form a basis over $\mathbb R$ for the Hermitian matrices).    

  Physicists write $\vec \sss = (\sss^x, \sss^y, \sss^z) $ and   
 \bc $\vec \sss_k= \mathbb I \otimes \ldots \otimes \mathbb I  \otimes \vec \sss \otimes  \mathbb I  \otimes \ldots \otimes  \mathbb I $, \ec where $\mathbb I = \left(\begin{matrix} 1 & 0 \\ 0& 1\end{matrix}\right)$ , and  the $\vec \sss$ is in position $k$.  
   
   The Hamiltonian is now a Hermitian operator on $(\mathbb C^2)^{\otimes n}$:
\bc   $H= \sum_{i=1}^{n-1} h^{(2)}_{i,i+1}$  where    $h^{(2)}_{i,i+1}= \frac J 4 (\vec \sss_i \cdot \vec \sss_{i+1} -\mathbb I^{\otimes n})$. \ec 
 $J \in \mathbb R$ is a coupling constant, and the local Hamiltonians 
 $h^{(2)}_{i,i+1}$  describe the interaction of neighbouring sites.   Note that this only depends on $i$ as far as the subspace the operator acts on nontrivially is concerned. The action on that subspace remains the same for each $i$.


  \subsection*{Abstract spin chains}
  For $d\ge 2$ (sometimes suppressed), a qu\emph{d}it is a unit vector in $d$-dimensional Hilbert space $\mathbb C^d$.    
  An \emph{abstract  spin chain} is a  system of $n$ qudits,   arranged linearly. The positions are referred to as \emph{sites}. 
  The state of such a system is given by a vector in the $d^n$-dimensional Hilbert space $(\mathbb C^d)^{\otimes n}$.      
 One also considers higher dimensional arrangements of qudits, e.g.\ square lattices.



  A Nature paper of      Cubitt, Perez-Garcia and Wolf~\cite{Cubitt.etal:15} showed that  whether there is a spectral gap is undecidable for the square lattice (2D) case.  
    The full proof has last been updated on arXiv in July  2020 (1502.04573v4).

Bausch, Cubitt, Lucia and Perez-Garcia~\cite{Bausch.etal:20} showed that the existence of a spectral gap is undecidable  for  the spin  chain (1D) case.

 As in the case of the  Ising and Heisenberg chains, the behaviour of    an abstract  spin chain is described by local Hamiltonians.   $M_n(\mathbb C)$ denotes the algebra of $n\times n$ complex matrices.
  Let $h^{(1)}\in M_d(\mathbb C)$ and $h^{(2)}\in M_{d^2}(\mathbb C)$ be   Hermitian matrices, where 
\bi \item $h^{(1)}$ describes the one-site ``interactions", and 
  
 \item  $h^{(2)}$ describes the nearest-neighbour interactions. \ei

      The  global Hamiltonian of a spin chain of $n$ qudits is   given by shifting and adding up these interactions as the indices vary:
 \[H_n= \sum_{i=1}^n h^{(1)}_i + \sum_{i=1}^{n-1} h^{(2)}_{i,i+1}.\]  
%


  \subsection*{Spectral gap}
The spectral gap of a Hamiltonian $H$ acting on a finite-dimensional Hilbert space is $\Delta(H)=\gamma_1(H)-\gamma_0(H)$, the difference between its least two eigenvalues.

 $\seq{H_n}\sN n$ will always denote a sequence such that $H_n$ is a    Hamiltonian on   the $d^n$-dimensional Hilbert space.     
   The   asymptotic spectral gap of such a sequence  can be defined as  \bc $\Delta\seq{H_n}=\liminf_n \Delta(H_n).$\ec    
  (Note that     the ground energy $\lambda_0(H_n) $ typically   increases with~$n$.) Naively  one would think  that 
  the  system is called gapped if $\Delta\seq{H_n}$ is positive, and gapless otherwise.  
Cubitt et al.~\cite{Cubitt.etal:15} and then Bausch et al.~\cite{Bausch.etal:20} use   definitions making both the gapped and the gapless case more restricted, so that some  sequences have neither property.

  \begin{definition}     $\seq{H_n}$ is \emph{gapped} if   $\Delta \seq {H_n}=\liminf_n \Delta(H_n)$ is positive, and moreover,     for sufficiently large $n$, the least eigenvalue $\lambda_0(H_n)$ is non-degenerate, i.e.\ its eigenspace has dimension $1$.    \end{definition}

The second condition means that there is a unique ground state of the system (up to phase).

  \begin{definition}       $\seq{H_n}$ is \emph{gapless} if there is some $c>0$ such that for each $\varepsilon >0$,  for sufficiently large $n$, each point in  the interval $[\lambda_0(H_n), \lambda_0(H_n)+c]$ is $\varepsilon$-close to some eigenvalue of $H_n$.    \end{definition}


    
    The following figure {taken from Cubitt et al.~\cite{Cubitt.etal:15}  demonstrates the two cases in the thermodynamic limit. 
\bc \scalebox{.4}{\includegraphics{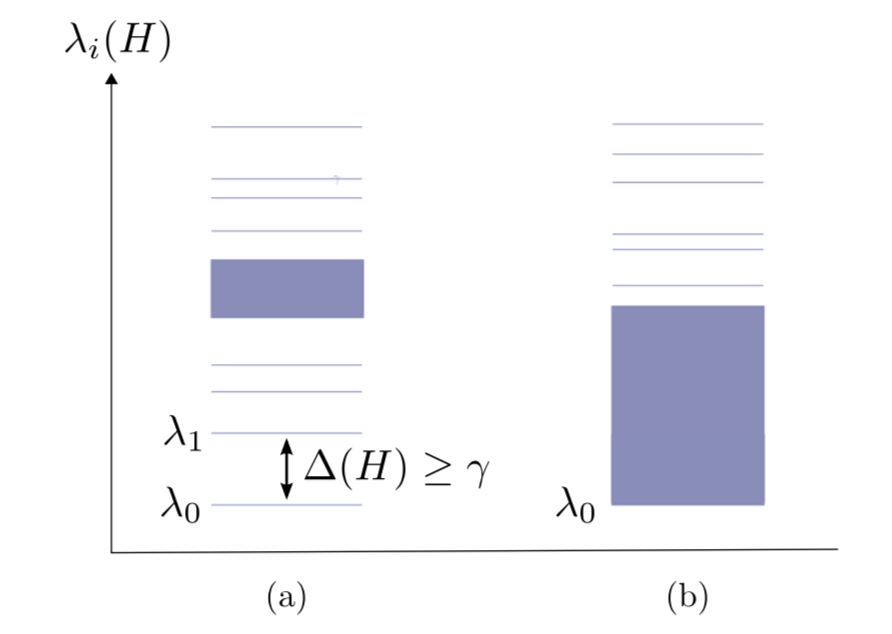}}\ec

In  the 1-dimensional case, whether there is a spectral gap was    shown to be undecidable by Bausch et al.~\cite{Bausch.etal:20}. 
  Given a Turing machine $M$,  they determine a (large)  dimension $d$.  
Then, given an   input $\eta \in \NN$ to $M$ they compute local Hamiltonians  $h^{(1)}\in M_d(\mathbb C)$ and $h^{(2)}\in M_{d^2}(\mathbb C)$ as above such that \bi \item  if  $M(\eta)$ halts then the sequence  $\seq{H_n(\eta) }\sN n$  (defined as above by shifting the local interactions) is gapless,  \item  otherwise the sequence $\seq{H_n(\eta)}\sN n$ is gapped.\ei

  They rely on the methods Cubitt et al.~\cite{Cubitt.etal:15}   who had shown earlier on  that the spectral gap problem is undecidable in the 2D case, using    square lattices of qudits. The definitions, in particular the Hamiltonians in the 2D case, are similar to the ones given here. However,  there are two types of nearest-neighbour interactions, corresponding to rows and columns.     
 Interestingly,  in the 2D case the relationship between machines and Hamiltonians is   the other way round:   if  $M(\eta)$ halts then the sequence is gapped, else gapless.

 The entries of the Hamiltonians are easy ``complex" numbers:    Let $F$ be the subring of $\mathbb C$ generated by   $\QQ \cup \{\sqrt 2\} \cup \{ \exp( 2 \pi i \theta) \colon \theta \in \QQ\}$.   The entries of the local Hamiltonians, and hence of the $H_n(\eta)$,  are all in $F$.   So the undecidability of the spectral gap is not an artefact of the  well-known fact that equality of  two computable reals is undecidable.

Here are some further elements of the proofs.  
The 2D case 
relies on  quantum Turing machines (Bernstein and Vazirani), and  the  history state  Hamiltonian due to    Feynman,    then Kitaev, then Gottesman and Irani (FOCS 2013):
The ground state of such a Hamiltonian  encodes the whole computation  of a QTM up to a stage $T$.  
Note that  the QTM is not related to $M$; rather,  it is related to the    phase estimation algorithm (see e.g. \cite{Nielsen.Chuang:02}).    
  Quasi-periodic Wang tilings  due to Robinson (Inventiones, 1971) also play an important  role.    
In  the 1D-case   the Wang tiling (which needed the second spatial dimension in the lattice setting) is replaced by  ``marker Hamiltonians".
There are various articles available providing an overview of the proof in the 2D case, e.g. the Nature paper~\cite{Cubitt.etal:15} itself, and also  \cite{Kreinovich:17}.

%

%

\def\cprime{$'$} \def\cprime{$'$}

%
%

\end{document}